\documentclass[12pt,leqno,amsfonts,amscd]{amsart}
\usepackage{epsfig}
\usepackage{amsmath}
\usepackage{amsfonts}
\usepackage{amssymb}
\usepackage{color}
\usepackage{hyperref}

\newtheorem{theorem}{Theorem}[section]
\newtheorem{prop}[theorem]{Proposition}

\theoremstyle{definition}
\newtheorem{definition}[theorem]{Definition}
\newtheorem{example}[theorem]{Example}

\theoremstyle{remark}

\newtheorem{remark}[theorem]{Remark}

\numberwithin{equation}{section}

\newcommand{\NN}{{\mathbb N}}

\newcommand{\CC}{{\mathbb C}}

\newcommand{\out}[1]{\ }

 \DeclareMathOperator{\NP}{NP}

\let\PSH=\psh

\let\cal=\mathcal
\renewcommand{\phi}{\varphi}

\begin{document}

\title[Remarks on weak convergence of complex Monge-Amp\`ere measures]{Remarks on
weak convergence of complex Monge-Amp\`ere measures}

\author[M. El Kadiri]{Mohamed El Kadiri}
\address{University of Mohammed V
\\Department of Mathematics,
\\Faculty of Science
\\P.B. 1014, Rabat
\\Morocco}
\email{elkadiri@fsr.ac.ma}



\subjclass[2010]{31C10, 32U05, 32U15.}
\keywords{Plurisubharmonic function, Plurifine topology, Plurifinely open set,
Monge-Amp\`ere operator, Monge-Amp\`ere measure.}

\begin{abstract}
Let $(u_j)$ be a
deaceasing
sequence of
psh functions in the domain of definition $\cal D$ of the
Monge-Amp\`ere operator on a domain $\Omega$
of $\mathbb{C}^n$ such that $u=\inf_j u_j$ is plurisubharmonic on $\Omega$.
In this paper we
are interested in the problem of finding
conditions insuring that
\begin{equation*}
\lim_{j\to +\infty} \int\varphi (dd^cu_j)^n=\int\varphi  \rm{NP}(dd^cu)^n
\end{equation*}
for any continuous function on $\Omega$ with compact support, where $\rm{NP}(dd^cu)^n$ is the
nonpolar part of $(dd^cu)^n$, and conditions
implying that $u\in \cal D$. For $u_j=\max(u,-j)$ these conditions
imply also that
\begin{equation*}
\lim_{j\to +\infty} \int_K(dd^cu_j)^n=\int_K \rm{NP}(dd^cu)^n
\end{equation*} 
for any
compact set $K\subset\{u>-\infty\}$.
\end{abstract}
\maketitle

\section{Introduction}

Let $\Omega$ be a domain of $\CC^n$. In \cite{BT2} Bedford and Taylor proved that
if $u$ and $v$ are
locally bounded plurisubharmonic (psh in abbreviated form) functions on $\Omega$
such that $u=v$ on a plurifinely open set $\cal O\subset \Omega$ then
$(dd^cu)^n|_{\cal O}=(dd^cv)^n|_{\cal O}$. This result allowed them
to define the non-polar part $\NP(dd^cu)^n$ of the Monge-Amp\`ere
measure of $u$, a Borel measure on $\Omega$ which does not put
masses on pluripolar sets, see Section \ref{section3}. Then they stated the following result:
Let $u$ be a psh function  on $\Omega$ and  $K\subset  \{u > -\infty \}$ be
a given compact subset of $\Omega$. If $\{u_j\}$  is a
sequence of locally bounded psh functions on a $\Omega$ decreasing to
$u$, then
$$\lim_{j\to +\infty} \int_K (dd^cu_j)^n = \int_K \NP(dd^cu)^n$$
(see \cite[Proposition 4.4]{BT2}). Unfortunately, this result is false as we shall see in
Section \ref{section3} by giving a counter-example. This leads us  to look for
a sufficient condition on the function $u$ which implies that the above result is true.
The found condition implies also that if $u$ is finite, then $u$ is in the general domain
of definition of the Monge-Amp\`ere operator if and only if the Borel measure $\NP(dd^cu)^n$
is indeed a Radon measure on $\Omega$. We also prove that if this  condition
holds and  if $u$ is finite and locally maximal,
then $u$ is maximal.

We shall use the plurifine topology on the open set $\Omega$. This topology is defined
as being the coarsest one on $\Omega$
that makes continuous  all  plurisbharmonic functions on $\Omega$.
An open set for this topology is called a plurifinely open set.
The plurifine topology has been investigated by many authors,
namely, Bedford and Taylor, El Marzguioui, El Kadiri, Fuglede and Wiegerinck, see \cite{BT2},
\cite{EK1}, \cite{EKFW} and \cite{EMW}.
For more details on the properties of this topology we refer the reader
to \cite{BT2}, \cite{EKFW} and \cite{EMW}.

\section{The cegrell classes and the domain of definition of the
Monge-Amp\`ere operator}\label{section2}
Let $\Omega$ be a bounded hyperconvex domain in $\CC^n$. From \cite{Ce1} and \cite{Ce2} we recall
the following subclasses of $\PSH_-(\Omega)$, the cone of nonpositive plurisubharmonic functions
on $\Omega$:

$${\cal E}_0={\cal E}_0(\Omega)=\{\varphi\in \PSH_-(\Omega): \lim_{z\to \partial \Omega}\varphi(z)=0, \ \int_\Omega (dd^c\varphi)^n<\infty\},$$

$${\cal F}={\cal F}(\Omega)=\{\varphi\in \PSH_-(\Omega): \exists \ \cal E_0\ni \varphi_j\searrow \varphi, \
\sup_j\int_\Omega (dd^c\varphi_j)^n<\infty\},$$

and
\begin{equation*}
\begin{split}
\cal E=\cal E(\Omega)=\{\varphi \in \PSH_-(\Omega): \forall z_0\in \Omega, \exists
\text{ a neighborhood } \omega \ni z_0,&\\
{\cal E}_0\ni \varphi_j \searrow \varphi \text{ on } \omega, \ \sup_j \int_\Omega(dd^c\varphi_j)^n<\infty \}.
\end{split}
\end{equation*}

As in \cite{Ce1}, we note that if $u\in \PSH_-(\Omega)$ then $u\in \cal E(\Omega)$ if and only
if for every $\omega\Subset \Omega$, there is $v\in \cal F(\Omega)$ such that $v\ge u$ and
$v=u$ on $\omega$. On the other hand we have $\PSH_-(\Omega)\cap L^\infty_{loc}(\Omega)\subset \cal E(\Omega)$.
The classical Monge-Amp\`ere operator on $\PSH_-(\Omega)\cap L^\infty_{loc}(\Omega)$ can be extended uniquely to
the class $\cal E(\Omega)$, the extended operator is still denoted by $(dd^c\cdot)^n$.
According to Theorem 4.5 from \cite{Ce1}, the class $\cal E$ is the biggest class
$\cal K\subset \PSH_-(\Omega)$ satisfying the following conditions:

(1) If $u\in \cal K$, $v\in \PSH_-(\Omega)$ then $\max(u,v)\in \cal K$.

(2) If $u\in \cal K$, $\varphi_j\in \PSH_-(\Omega)\cap L^\infty_{loc}(\Omega)$, $\varphi_j\searrow u$, $j\to +\infty$,
then $((dd^c\varphi_j)^n)$ is weak*-convergent.

We also recall, following Blocki, cf \cite{Bl1}, that the general domain of definition $\cal D$ of
the Monge-Amp\`ere operator on a domain $\Omega$ of $\CC^n$ consists of
pluri-subharmonic functions $u$
on $\Omega$ for which there is a nonnegative (Radon) measure $\mu$ on $\Omega$ such that for any decreasing sequence
$(u_j)$ of smooth pluri-subharmonic functions on $\Omega$ converging to $u$, the sequence of measures $(dd^cu_j)^n$ is weakly*-convergent
to $\mu$. The measure $\mu$ is denoted by $(dd^cu)^n$ and called the Monge-Amp\`ere
measure of (or associated with) $u$.
When $\Omega$ is bounded and hyperconvex then $\cal D\cap\PSH_-(\Omega)$ coincides with the class $\cal F=\cal F(\Omega)$, cf. \cite{Bl1}.

\section{The nonpolar part of the Monge-Amp\`ere measure of a
plurisubharmonic function}\label{section3}

Let us first recall the following result of Bedford and Taylor:

\begin{prop}[{\cite[Corollary 4.3]{BT2}}]\label{prop3.1} Let $u$ and $v$ be two locally bounded
psh functions on a domain $\Omega\subset \CC^n$
and $U$ a plurifinely open set $\subset \Omega$. If
$u=v$ on $U$, then the restrictions of the Monge-Amp\`ere measures $(dd^cu)^n$
and $(dd^cv)^n$ to $U$ are equal.
\end{prop}

In \cite{EK2} El Kadiri has extended this result to the
general case where $u,v$ are psh functions in the general
domain of definition $\cal D$ of the Monge-Amp\`ere operator on
$\Omega$, see \cite[Corollary 3.7]{EK2}.\\

Proposition \ref{prop3.1} allows one to define the nonpolar part of $(dd^cu)^n$
for a psh function $u$ on $\Omega$ as follows:

\begin{definition}\label{def3.2}
For a function $u \in \PSH(\Omega)$,
the nonpolar part $\NP(dd^cu)^n$ of $(dd^cu)^n$
is the (Borel) measure which is zero on the pluripolar
set $\{u=-\infty\}$, and for a Borel set $E\subset \{u>-\infty\}$,
$$\int_E\NP(dd^cu)^n=\lim_{j\to +\infty}\int_{E\cap\{u>-j\}}(dd^c\max\{u,-j\})^n.$$
\end{definition}

The limit exists and the definition makes a sense because the numerical sequence $(\int_{E\cap\{u>-j\}}(dd^c\max\{u,-j\})^n)$
is increasing according to Proposition \ref{prop3.1}. If $u$ is locally bounded,
then $\NP(dd^cu)^n=(dd^cu)^n$. Indeed, let $k$ be a fixed integer and $E$
a Borel subset of $\Omega$. For any integer $j>k$ we have
$u=\max\{u,-j\}$ on the plurifinely open set $\{u>-k\}$ and
$\{u>-k\}\subset \{u>-j\}$, and hence
\begin{align*}
\int_{E\cap \{u>-k\}}&\NP(dd^cu)^n =\lim_{j\to +\infty}\int_{(E\cap \{u>-k\})\cap\{u>-j\}}(dd^c\max\{u,-j\})^n\\
=& \lim_{j\to +\infty}\int_{E\cap \{u>-k\}}(dd^c\max\{u,-j\})^n=\int_{E\cap \{u>-k\}}(dd^cu)^n
\end{align*}
by Proposition \ref{prop3.1}. By letting $k\to +\infty$, this leads to
$$\int_E\NP(dd^cu)^n=
\int_E(dd^cu)^n.$$

\begin{prop}\label{prop3.2}
Let $u\in \cal D$. Then $\NP(dd^cu)^n=1_{\{u>-\infty\}}(dd^cu)^n$. In particular, $\NP(dd^cu)^n$
is a Radon measure on $\Omega$.
\end{prop}

\begin{proof}
Indeed, for any compact set $K\subset \Omega$ we have
\begin{eqnarray*}
\NP(dd^cu)^n(K)&=& \lim_{j\to +\infty}\int_{K\cap \{u>-j\}}(dd^c\max\{u,-j\})^n\\
&=& \lim_{j\to +\infty}\int_{K\cap \{u>-j\}}(dd^cu)^n
=\int_{K\cap \{u>-\infty\}}(dd^cu)^n (<+\infty)
\end{eqnarray*}
in view of the definition of $\NP(dd^cu)^n$, Theorem 3.5 from \cite{EK2} and the
monotone convergence theorem. It then follows that $\NP(dd^cu)^n=1_{\{u>-\infty\}}(dd^cu)^n$ and
that $\NP(dd^cu)^n$ is a Radon measure on $\Omega$.
\end{proof}

\begin{example}\label{example3.2}
For every integer $j>0$ set
$$u_j(z)=\max(\log\|z\|, \frac{1}{j})$$
and $$u(z)=\max(\log\|z\|,0)$$
for every $z\in \CC^n$. Then $(u_j)\searrow u$ on $\CC^n$.
The functions $u_j$, $j=1,2,...$, and $u$ are locally bounded psh functions
on $\CC^n$. Let $K=\{z\in \CC^n: \|z\|\leq 1\}$. Then $K$ is a compact subset of
$\CC^n$ and we have
$\int_K(dd^cu_j)^n=0$ for every $j\geq 1$, so that
$\lim_{j\to +\infty}\int_K (dd^cu_j)^n=0$ while $\int_K\NP(dd^cu)^n=\int_K(dd^cu)^n>0$.
This example shows clearly that  Proposition 3.4 from \cite{BT2} is false and leads us to find
sufficient and/or necessary conditions on $u$ and the compact $K$ for which this proposition
is true when $u_j=\max\{u,-j\}$.
\end{example}

\begin{theorem}\label{thm3.4}
Let $u$ be a psh function on $\Omega$
and $K$ a compact subset of $\Omega$.
Suppose that $\NP(dd^cu)^n(K)<+\infty$. Then the following assertions are equivalent

{\rm (a)}$$\lim_{j\to +\infty} \int_K(dd^c\max\{u,-j\})^n=\int_K\NP(dd^cu)^n$$

{\rm (b)}
$$\lim_{j\to +\infty}\int_{K\cap \{u=-j\}}(dd^c\max\{u,-j\})^n=0.$$
\end{theorem}

\begin{proof}
Indeed, for every integer $j$ we have
\begin{align*}\int_K(dd^c\max\{u,&-j\})^n\\
&=\int_{K\cap \{u>-j\}}(dd^c\max\{u,-j\})^n+
\int_{K\cap \{u\leq-j\}}(dd^c\max\{u,-j\})^n\\
&=\int_{K\cap \{u>-j\}}(dd^c\max\{u,-j\})^n+
\int_{K\cap \{u=-j\}}(dd^c\max\{u,-j\})^n,
\end{align*}
because $(dd^c\max\{u,-j\})^n=0$ on $\{u<-j\}$ by Proposition \ref{prop3.1}, and hence
$$\lim_{j\to +\infty} \int_K (dd^c\max\{u,-j\})^n = \int_K \NP(dd^cu)^n$$
if and only if
$$\lim_{j\to +\infty}\int_{K\cap \{u=-j\}}(dd^c\max\{u,-j\})^n=0.$$
\end{proof}

\begin{remark}\label{remark3.5} If we drop the condition `$\NP(dd^cu)^n(K)<+\infty$' in Theorem \ref{thm3.4} we only have the implication
(b) $\Rightarrow$ (a).
\end{remark}

In what follows we denote by $C_n$ the capacity of Bedford and Taylor on $\Omega$,
see \cite[Definition 3.1]{BT1}. Recall that $C_n$ is defined by
$$C_n(K)=C_n(K,\Omega):=\sup\{\int_\Omega (dd^cu)^n: u\in \PSH_-(\Omega), \ u\leq -1 \text{ on } K\}.$$
Then $$C_n(K)=\int_K(dd^cu_K^*)^n$$
for any compact set $K\subset \Omega$, where
$$u_K=\sup\{u\in \PSH_-(\Omega):  u\leq -1 \text{ on } K\}, $$
and where
$$u_K^*(z)=\limsup_{\zeta\to z}u_K(\zeta)$$
for every $z\in \Omega$, see \cite[Proposition 5.3 (ii)]{BT1}.

\begin{theorem}\label{thm3.6}
Let $u$ be a psh function on $\Omega$
and $K$ a compact subset of $\Omega$ such that $K\subset \{u>-\infty\}$.
Suppose that $\lim_{j\to +\infty}j^nC_n(\{u=-j\})=0$. Then
$$\lim_{j\to +\infty} \int_K (dd^c\max\{u,-j\})^n = \int_K \NP(dd^cu)^n.$$
\end{theorem}

\begin{proof}
By adding a suitable constant to $u$ and reducing if necessary the domain $\Omega$
we may suppose that $u<0$ on $\Omega$.
For every integer $j>0$ we have

\begin{align*}
\int_{K\cap \{u>-j\}}(&dd^c\max\{u,-j\})^n \leq  \int_K(dd^c\max\{u,-j\})^n\\
&= \int_{K\cap \{u>-j\}}(dd^c\max\{u,-j\})^n
+  \int_{K\cap \{u=-j\}}(dd^c\max\{u,-j\})^n\\
&\leq  \int_K \NP(dd^cu)^n+j^n C_n(\{u=-j\}).
\end{align*}
By letting $j\to \infty$ this leads to the equality in the theorem.
\end{proof}

As an application of Theorem \ref{thm3.6} we have the following result:

\begin{theorem}\label{thm3.7}
Let $u$ be a finite psh function on $\Omega$ such that
$u$ is locally maximal and $\lim_{j\to +\infty}j^n C_n(\{u=-j\})=0$.
Then $u$ is maximal.
\end{theorem}

\begin{proof}
By \cite[Theorem 4.15]{EKS} we have $\NP(dd^cu)^n=(dd^cu)^n=0$.
Let $\varphi$ be a continuous real function on $\Omega$
with compact support $K$, then we have by Theorem \ref{thm3.6}
\begin{align*}
\limsup_{j\to +\infty}\int\varphi&(dd^c\max\{u,-j\})^n\\
&\leq \|\varphi\|_\infty \lim_{j\to +\infty}\int_K(dd^c\max\{u,-j\})^n
=\|\varphi\|_\infty \NP(dd^cu)^n(K)=0.
\end{align*}
It then follows that the measures $(dd^c\max\{u,-j\})^n$, $j\in \NN$,  converge
weakly to 0 and hence $u$ is maximal according to \cite[Theorem 4.4]{Bl2}.
\end{proof}

\section{Weak convergence of Monge-Amp\`ere measures}

\begin{theorem}\label{thm3.8}
Let $u$ be a nonpositive psh function on $\Omega$ and suppose that
$$\lim_{j\to +\infty}j^nC_n(\{u\leq -j\})=0.$$ Suppose moreover
that $\NP(dd^cu)^n$ is a Radon measure (that is, finite on
compact subsets of $\Omega$). Then for any decreasing
sequence $(u_k)$ of nonpositive psh functions with limit $u$, the sequence $(dd^cu_k)^n$
converge weakly to $\NP(dd^cu)^n$. In particular,
$u\in \cal D$ and $(dd^cu)^n=\NP(dd^cu)^n$.
\end{theorem}

\begin{proof}
Let $j$ be an integer and $\varphi$ a nonnegative continuous
function on $\Omega$ with compact support.
For every $k\in \NN$, we have in view of Proposition \ref{prop3.1}
\begin{eqnarray*}
\int\varphi (dd^cu_k)^n &=& \int_{\{u_k>-j\}}\varphi (dd^cu_k)^n+
\int_{\{u_k\leq -j\}}\varphi(dd^cu_k)^n\\
&=& \int_{\{u_k>-j\}}\varphi (dd^c\max\{u_k,-j\})^n+
\int_{\{u_k\leq -j\}}\varphi(dd^cu_k)^n\\
&\leq &\int\varphi(dd^c\max\{u_k,-j\})^n+\int_{\{u_k\leq -j\}}\varphi(dd^cu_k)^n.
\end{eqnarray*}
The first term of the second member of the last inequalty converges to
$$\int \varphi(dd^c\max\{u,-j\})^n$$ as $k\to +\infty$.
On the other hand, since $u_k\leq 0$ and $u\leq u_k$, we have
\begin{eqnarray*} 0 &\leq& \int_{\{u_k\leq -j\}}\varphi (dd^cu_k)^n
\leq \|\varphi\|_\infty\int_{\{u_k\leq -j\}} (dd^cu_k)^n\\
&\leq& \|\varphi\|_\infty j^nC_n(\{u_k\leq -j\})\leq \|\varphi\|_\infty j^nC_n(\{u\leq -j\})
\end{eqnarray*}
and
\begin{align*}
\int&\varphi (dd^c\max\{u,-j\})^n\\
&\leq  \int_{\{u>-j\}}\varphi (dd^c\max\{u,-j\})^n+ \|\varphi\|_\infty\int_{\{u\leq -j\}} (dd^c\max\{u,-j\})^n\\
&\leq  \int_{\{u>-j\}}\varphi (dd^c\max\{u,-j\})^n+ \|\varphi\|_\infty j^nC_n(\{u\leq-j\}.
\end{align*}
Hence
\begin{align*}
\limsup_{k\to +\infty}&\int\varphi (dd^cu_k)^n\\
&\leq \int_{\{u>-j\}}\varphi (dd^c\max\{u,-j\})^n+ 2j^n\|\varphi\|_\infty C_n(\{u=-j\},
\end{align*}
and then, by letting $j\to \infty$,
\begin{equation}\label{eq1}
\limsup_{k\to +\infty}\int\varphi (dd^cu_k)^n\leq \int\varphi \NP(dd^cu)^n.
\end{equation}
Let us now  again fix $j\in \NN$, we have $u_k=\max\{u_k,-j\}$ on the
plurifinely open set $\{u>-j\}$, hence
$$\int \varphi(dd^cu_k)^n\geq \int_{\{u>-j\}}\varphi(dd^c\max\{u_k,-j\})^n$$
in view of Proposition \ref{prop3.1}.
According to \cite[Theorem ]{Kl}, the plurifine topology is regular. It then follows
from the quasi-Lindel\H{o}f property of the plurifine topology that
the plurifinely l.s.c. function $1_{\{u>-j\}}$ is q.e. the upper envelop of an
increasing sequence $(\chi_l)$ of nonnegative plurifinely continuous
functions with plurifine supports  relatively compact in $\Omega$. Thus we have
$$\int\varphi (dd^cu_k)^n\geq \int\varphi \chi_l(dd^c\max\{u_k,-j\})^n$$
for every $l$. By fixing $l$ and letting $k\to \infty$, we obtain
$$\liminf_{k\to+\infty} \int\varphi (dd^cu_k)^n\geq \int\varphi\chi_l(dd^c\max\{u,-j\})^n$$
according to \cite[Theorem 3.2 (3)]{BT1}. By
letting $l\to \infty$ this leads  to
$$
\liminf_{k \to +\infty} \int\varphi (dd^cu_k)^n\geq \int\varphi 1_{\{u>-j\}}(dd^c\max\{u,-j\})^n
$$
since the measures $(dd^c\max\{u,-j\})^n$ put no mass in pluripolar sets.
Next, by letting $j\to +\infty$, we finally obtain
$$ \liminf_{k\to +\infty} \int\varphi (dd^cu_k)^n\geq \int\varphi \NP(dd^cu)^n.$$
The latter inequality combined with (\ref{eq1})gives the desired result.
\end{proof}

\begin{remark}\label{remark3.9}
1. It follows from the proof that the inequality
$$\liminf_{k\to +\infty} \int \varphi(dd^cu_k)^n\geq \int \varphi \NP(dd^cu)^n$$
is true for every $u\in \PSH(\Omega)$.

2. The proof also shows that Theorem \ref{thm3.8} is true if the condtion
$$\lim_{j\to +\infty}j^nC_n(\{u\leq -j\}=0$$
is replaced
by the more general condition $$\limsup_{j\to +\infty}\sup_k\int_{\{u_k\leq -j\}}(dd^cu_k)^n=0.$$
\end{remark}

\begin{remark}\label{remark3.10}
The converse of theorem \ref{thm3.8} is not true. Indeed, take $n=1$ and let $\Omega$ be
the domain subset $B(0,1)\setminus\{0\}$ of $\CC$ and let $u$ be the (finite) psh
function on $\Omega$ defined by $u(z)=\log |z|$. We have $\NP(dd^cu)=dd^cu=0$ on $\Omega$,
and hence, for any decreasing
sequence $(u_k)$ of non-positive psh functions decreasing to $u$, the sequence $(dd^cu_k)$
converge weaky to $\NP(dd^cu)$. On the other hand we have
$C_1(\{u\leq -j\})=C_1(B(0,e^{-j}), B(0,1))=\frac{2\pi}{j}$  for every $j>1$, so that
$\lim_{j\to+\infty}j C_1(\{u\leq -j\})=2\pi\ne 0$.
\end{remark}

\begin{remark}\label{remark3.11}
Let $u$ be a psh function on $\Omega$. The condition
$$\lim_{j\to +\infty}j^nC_n(\{u\leq -j\})=0$$
does not imply in general (even if $\NP(dd^cu)^n$ is a Radon measure) that
$$\lim_{k\to +\infty}\int_K(dd^cu_k)^n=\int_K \NP(dd^cu)^n$$
for any decreasing
sequence $(u_k)$  of locally bounded psh functions with limit $u$ and for any
compact set $K\subset \{u>-\infty\}$. Indeed, let $(u_k)$ be the
decreasing sequence of psh functions on $\CC^n$
given in the example \ref{example3.2} and let $u(z)=\lim_{k\to \infty}u_k(z)$
for $z\in \CC^n$. Then $u$ is a locally bounded psh function on $\CC^n$ and
$(dd^cu)^n$ is non-null and carried by the compact $K=\{\|z\|=1\}$. We have
$(dd^cu_k)^n(K)=0$ for any $k$ and $\NP(dd^cu)^n=(dd^cu)^n$ because $u$ is locally bounded,
so that
$$\lim_{k\to +\infty}\int_K(dd^cu_k)^n=0\ne \int_K\NP(dd^cu)^n.$$
\end{remark}

\begin{theorem}\label{thm3.12}
Let $u$ be a finite psh function on $\Omega$
such that $\lim_{j\to +\infty}j^nC_n(\{u\leq -j\})=0$. Then $u\in \cal D$ if and
only if $\NP(dd^cu)^n(K)<+\infty$ for any compact set $K\subset \Omega$
(that is, $\NP(dd^cu)^n$ is a Radon measure on $\Omega$).
In the affirmative case, one has $(dd^cu)^n=\NP(dd^cu)^n$.
\end{theorem}

\begin{proof}
The `only if' part is Proposition \ref{prop3.2}. For the `if part', suppose that  $\NP(dd^cu)^n(K)<+\infty$ for any compact set $K\subset \Omega$
and let $(u_k)$ be a decreasing sequence of smooth
psh functions on $\Omega$, converging to $u$
Then, according to Theorem \ref{thm3.8}, the sequence of the measures $(dd^cu_k)^n$
converges weakly to the Radon measure $\NP(dd^u)^n$
on $\Omega$. Hence $u\in \cal D$ by Theorem 1.1 from \cite{Bl1}.
\end{proof}

{\bf Acknowledgements.} The author would like to thank
N.X. Hong for having kindly communicated
to him the example given in Section 2 of the present paper.

\end{document}